\documentclass[10pt]{amsart}
\usepackage {amsmath, amscd}
\usepackage {amssymb}
\usepackage {bm}
\usepackage{mathpazo}

\usepackage[active]{srcltx}

\def\<{\langle}
\def\>{\rangle}

\numberwithin{equation}{section}

\def\CC{{\mathcal C}}
\def\DD{{\mathcal D}}

\def\HH{{\mathcal H}}

\def\JJ{{\mathcal J}}

\def\TT{{\mathcal T}}

\def\VV{{\mathcal V}}

\def\bbC{\mathbb{C}}

\def\bbN{\mathbb{N}}

\def\bbD{\mathbb{D}}
\def\bbT{\mathbb{T}}

\usepackage{color}

\makeatletter
\def\moverlay{\mathpalette\mov@rlay}
\def\mov@rlay#1#2{\leavevmode\vtop{%
   \baselineskip\z@skip \lineskiplimit-\maxdimen
   \ialign{\hfil$\m@th#1##$\hfil\cr#2\crcr}}}
\newcommand{\charfusion}[3][\mathord]{
    #1{\ifx#1\mathop\vphantom{#2}\fi
        \mathpalette\mov@rlay{#2\cr#3}
      }
    \ifx#1\mathop\expandafter\displaylimits\fi}
\makeatother

\newcommand{\dist}{\mathop{\rm dist}}

\newcommand{\Cont}{\mathbf{C}(\bbT)}

\newcommand{\Wth}{{\Phi}}

\newcommand{\Sp}{{\mathfrak{s}}}

\newtheorem{lemma}{Lemma}[section]
\newtheorem{proposition}[lemma]{Proposition}
\newtheorem{theorem}[lemma]{Theorem}
\newtheorem{corollary}[lemma]{Corollary}

\theoremstyle{definition}

\title[Recent results on truncated Toeplitz operators]{Recent results on truncated Toeplitz operators}

\author[Chalendar]{Isabelle Chalendar}
\address{  Universit\'e Lyon 1, INSA de Lyon, Ecole Centrale de Lyon, CNRS, UMR 5208, Institut Camille Jordan, 43 boulevard du 11 novembre 1918,
 F-69622 Villeurbanne Cedex, France}
\email{chalendar@math.univ-lyon1.fr}

\author[Fricain]{Emmanuel Fricain}
 \address{Laboratoire Paul Painlev\'e, Universit\'e Lille 1, 59 655 Villeneuve d'Ascq C\'edex, France}
 \email{emmanuel.fricain@math.univ-lille1.fr}

\author[Timotin]{Dan Timotin}
	\address{Institute of Mathematics of the 
 Romanian Academy, PO Box 1-764, Bucharest 014700, Romania.}
	\email{Dan.Timotin@imar.ro}



\begin{document}
	
	\begin{abstract}
		Truncated Toeplitz operators are  compressions of Toe\-plitz operators on model spaces; they have received much attention in the last years. This survey article presents several recent results, which relate boundedness, compactness, and spectra  of these operators to properties of their symbols. We also connect these facts with properties of the natural embedding measures associated to these operators. 
	\end{abstract}
	
	\maketitle

	\maketitle

	\section{Introduction}
	Truncated Toeplitz operators on model spaces have been formally introduced by Sarason in \cite{Sa}, although some special cases have long ago appeared in literature, most notably as model operators for completely nonunitary contractions with defect numbers one and for their commutant. This new area of study has been recently very active and many open questions posed by Sarason in \cite{Sa} have now been solved. See \cite{BBK,Be,MR3257897,MR3203060,MR3162251,GR,MR3043586,MR2869112,MR2830601,MR2679022,MR2648079}. Nevertheless, there are still basic and interesting questions which remain mysterious. 
	
	The truncated Toeplitz operators live on the  model spaces $K_\Theta$, which are the closed invariant subspaces for the backward shift operator $S^*$ acting on the Hardy space $H^2$ (see Section 2 for precise definitions). Given a model space $K_\Theta$ and a function $\phi\in L^2=L^2(\bbT)$, the truncated Toeplitz operator $A_\phi^\Theta$ (or simply $A_\phi$ if there is no ambiguity regarding the model space) is defined on a dense subspace of $K_\Theta$ as the compression to $K_\Theta$ of multiplication by $\phi$. The function $\phi$ is then called a symbol of the operator. An alternate way of defining a truncated Toeplitz operator is by means of a measure; in case $\phi$ is bounded, then a possible choice of the defining measure for $A^\Theta_\phi$ is $\phi\,dm$ (with $m$ Lebesgue measure).

	Note that the symbol or the associated measure are never uniquely defined by the operator. From this and other points of view the truncated Toeplitz operators have much more in common with Hankel Operators than with Toeplitz operators. This point of view will be occasionally pursued throughout the paper. 
	
	We intend to survey  several recent results that are mostly scattered in the literature. They focus on the relation between the operator and the symbol or the measure. Obviously the nonuniqueness is a main issue, and in some situations it may be avoided by considering the so-called standard symbol of the operator. The properties under consideration are boundedness, compactness, and spectra.   Most of the results presented are known, and our intention is only to put them in context and emphasize their connections, indicating the relevant references. Part of the embedding properties of measures have not appeared explicitely in the literature, so some proofs are provided only where references seemed to be lacking.
	
The structure of the paper is the following. After a preliminary section with generalities about Hardy spaces and model spaces, we discuss in section~3 Carleson measures, first for the whole $H^2$ and then for model spaces. Truncated Toeplitz operators are introduced in Section~4, where one also discusses some boundedness properties. Section~5 is dedicated to compactness of truncated Toeplitz operators, and Section~6 to its relation to embedding measures. The last two sections discuss Schatten--von Neumann and spectral properties, respectively.

	\section{Preliminaries}
	For the content of this section, \cite{MR0268655} is a classical reference for general facts about Hardy spaces, while~\cite{MR827223} can be used  for Toeplitz and Hankel operators as well as for model spaces. 
	
	\subsection{Function spaces}
	
	Recall that the Hardy space $H^p$ of the unit disk $\bbD=\{z\in\bbC:\, |z|<1\}$ is the space of analytic functions $f$ on $\bbD$ satisfying $\|f\|_p<+\infty$, where 
	\[
	\|f\|_p=\sup_{0\leq r<1}\left(\int_0^{2\pi}|f(re^{it})|^p\frac{dt}{2\pi}\right)^{1/p},\qquad (1\leq p<+\infty).
	\]
	The algebra of bounded analytic functions on $\bbD$ is denoted by $H^\infty$. We denote also $H_0^p=zH^p$ and $H^p_-=\overline{z H^p}$. Alternatively, $H^p$ can be identified (via radial limits) to the subspace of functions $f\in L^p=L^p(\bbT)$ for which $\hat f(n)=0$ for all $n<0$. Here $\bbT$ denotes the unit circle with normalized Lebesgue measure $m$.

	In the case  $p=2$, $H^2$ becomes a Hilbert space with respect to the scalar product inherited from $L^2$ and given by 
	$$
	\langle f,g\rangle_2=\int_{\bbT}f(\zeta)\overline{g(\zeta)}\,dm(\zeta),\qquad f,g\in L^2.
	$$
	The orthogonal projection from $L^2$ to $H^2$ will be denoted by $P_+$. The space $H_-^2$ is precisely the orthogonal of $H^2$, and the corresponding orthogonal projection is $P_-=I-P_+$.

	The Poisson transform of a function $f\in L^1$ is 
	\begin{equation}\label{eq:poisson transform}
		\hat f(z) =\int_\bbT f(\xi) \frac{1-|z|^2}{|1-\xi\bar z|^2}\, d\xi, \quad z\in\bbD.
	\end{equation}

Suppose now $\Theta$ is an inner function, that is a function in $H^\infty$ whose radial limits are of modulus one almost everywhere on $\bbT$.	Its spectrum  is defined by 
\begin{equation}\label{eq:def spectrum of theta}
\Sp(\Theta):=\{\zeta\in\bbD: \liminf_{\lambda\in\bbD, \lambda\to\zeta}|\Theta(\lambda)|=0\}.
\end{equation}

	Equivalently, if $\Theta=BS$ is the decomposition of $\Theta$ into a Blaschke product and a singular inner function, then $\rho(\Theta)$ is the union between the closure of the limit points of the zeros of $B$ and the support of the singular measure associated to $S$. We will also define 
	\[
	\rho(\Theta)=\Sp(\Theta)\cap \bbT.
	\]
	
	We define the corresponding {\emph {shift-coinvariant subspace}} generated by $\Theta$ (also called \emph{model space}) by the formula $K_\Theta^p=H^p\cap \Theta \overline{H_0^p}$, where $1\leq p<+\infty$.  We will be especially interested in the Hilbert case, that is when $p=2$. In this case, we also denote by $K_\Theta=K_\Theta^2$ and it is easy to see that $K_\Theta$ is also given by the following 
	\[
	K_\Theta=H^2\ominus\Theta H^2=\left\{f\in H^2:\langle f,g\rangle=0,\forall g\in H^2\right\}.
	\]
	The orthogonal projection of $L^2$ onto $K_\Theta$ is denoted by $P_\Theta$. It is well known (see \cite[page 34]{MR827223}) that $P_\Theta=P_+-\Theta P_+\bar\Theta$. Since $P_+$ acts boundedly on $L^p$, $1<p< \infty$, this formula shows that $P_\Theta$ can also be regarded as a bounded operator from $L^p$ into 
	$K_\Theta^p$, $1<p<\infty$.

The spaces $H^2$ and $K_\Theta$ are reproducing kernel spaces over the unit disc $\bbD$. The respective reproducing kernels are, for $\lambda\in\bbD$,
\[
\begin{split}
k_\lambda(z)&= \frac{1}{1-\bar{\lambda} z},\\
k_\lambda^\Theta(z)&= \frac{1-\overline{\Theta(\lambda)}\Theta(z)}{1-\bar{\lambda} z}.
\end{split}
\]	

Evaluations at certain points $\zeta\in\bbT$ may also be bounded sometimes; this happens precisely when $\Theta$ has an angular derivative in the sense of Caratheodory at $\zeta$~\cite{AC}. In this case the function $k^\Theta_\zeta(z)=\frac{1-\overline{\Theta(\zeta)}\Theta(z)}{1-\bar{\zeta} z} $ is in $K_\Theta$, and it is the reproducing kernel for the point $\zeta$.
	 
	It easy to check that, if $f,g\in K_\Theta$, then $fg\in H^1\cap \bar z \Theta^2 H_-^1\subset K^1_{\Theta^2}$. In particular, if $f,g$ are also bounded, then $fg\in K_{\Theta^2}$. So $(k^\Theta_\lambda)^2\in K_{\Theta^2}$ for all $\lambda\in\bbD$.

The map $C_\Theta$ defined on $L^2$  by 
\begin{equation}\label{eq:conjugation}
C_\Theta f=\Theta \bar z\bar f;
\end{equation}
is a conjugation (i.e. $C_\Theta$ is anti-linear, isometric and involutive), which has the convenient supplementary property of mapping
$K_\Theta$ precisely onto $K_\Theta$. 
	
	\subsection{One-component inner functions}
	In view of their main role in the study of operators on model spaces, we devote this subsection to a particular class of inner functions.  Fix a number $0<\epsilon<1$, and define
	\begin{equation}\label{eq:definition of Omega}
	\Omega(\Theta,\epsilon)=\{z\in \bbD:|\Theta(z)|<\epsilon\}.
	\end{equation}
	The function $\Theta$ is called \emph{one-component} if there exists a value of $\epsilon$ for which  $\Omega(\Theta,\epsilon)$ is connected. (If this happens, then 	$\Omega(\Theta,\delta)$ is connected for every $\epsilon<\delta<1$.) One-component functions have been introduced by Cohn~\cite{Cohn1}. An extensive study of these functions appears in~\cite{AI23,A2}; all  results quoted below appear in~\cite{A2}.
	
	The above definition is not very transparent. In fact,  one-component functions are rather special: a first immediate reason is that they must satisfy $m(\rho(\Theta))=0$. This condition, of course, is not sufficient, but it suggests examining some simple cases. 
	
The set $\rho(\Theta)$ is empty for finite Blaschke products, which are one-component. 	The next simplest case is  when $\rho(\Theta)$ consists of just one point.  One can prove easily that the elementary singular inner functions
	$\Theta(z)=e^{\frac{z+\zeta}{z-\zeta}}$ (for $\zeta\in\bbT$) are indeed one-component. 
	
	Suppose then that $\Theta$ is a Blaschke product whose zeros $a_n$ tend nontangentially to a single point $\zeta\in \bbT$. If 
	\begin{equation}\label{eq:condition one-component}
		\inf_{n\ge 1}\frac{|\zeta-a_{n+1}|}{|\zeta-a_n|}>0,
	\end{equation}
	then $\Theta$ is one-component. So, in particular, if $0<r<1$ and $\Theta$ is the Blaschke product with zeros $1-r^n$, $n\geq 1$, then $\Theta$ is one-component. If condition~\eqref{eq:condition one-component} is not satisfied, then usually $\Theta$ is not   one-component. A detailed discussion of such Blaschke products  is given in~\cite{A2}, including the determination of the classes $\CC_p(\Theta)$ (see Subsection~\ref{subsection-carl-model-space}).

	One-component inner functions can be characterized by an estimate on  the $H^\infty$ norm of the reproducing kernels $k^\Theta_\lambda$. While for a general inner function $\Theta$ we have $\|k^\Theta_\lambda\|_\infty=O (1-|\lambda|^{-1})$, this estimate can be improved for one-component functions: $\Theta$ is one-component if and only if there exists a constant $C>0$ such that for every $\lambda\in\mathbb D$, we have 
	\[
	\|k^\Theta_\lambda\|_\infty \leq C \frac{1-|\Theta(\lambda)|}{1-|\lambda|}.
	\]


\subsection{Multiplication operators and their cognates}
	
	For $\phi\in L^\infty$, we denote by $M_\phi f=\phi f$ the multiplication operator on $L^2$; we have $\|M_\phi\|=\|\phi\|_\infty$. The Toeplitz operator $T_\phi: H^2\longrightarrow H^2$ and the Hankel operator $H_\phi: H^2\longrightarrow H^2_-=L^2\ominus H^2$ are given by the formulae
	$$
	T_\phi=P_+ M_\phi,\qquad H_\phi=P_-M_\phi.
	$$ 
	In the case where $\phi$ is analytic, $T_\phi$ is just the restriction of $M_\phi$ to $H^2$. We have $T_\phi^*=T_{\overline{\phi}}$ and $H_\phi^*=P_+M_{\overline{\phi}}P_-$. 
	
	It should be noted that, while the symbols of $M_\phi$ and $T_\phi$ are uniquely defined by the operators, this is not the case with $H_\phi$. Indeed, it is easy to check that $H_\phi=H_{\psi}$ if and only if $\phi-\psi\in H^\infty$. So statements about Hankel operators often imply only the existence of a symbol  with corresponding properties.

	The Hankel operators have the range and domain spaces different. It is sometimes preferable to work with an operator acting on a single space. For this, we introduce in $L^2$ the unitary symmetry $\JJ$ defined by
	\[
	\JJ(f) (z)=\bar z f(\bar z).
	\]
	We have then $\JJ(H^2)=H^2_-$ and $\JJ(H^2_-)=H^2$. Define $\Gamma_\phi:H^2\to H^2$ by
	\begin{equation}\label{eq:def of Gamma}
		\Gamma_\phi=\JJ H_\phi.
	\end{equation}
	Obviously properties of boundedness or compactness are the same for $H_\phi$ and $\Gamma_\phi$.
	
	The definition of $M_\phi$, $T_\phi$ and $H_\phi$ can be extended to the case when the symbol $\phi$ is only in $L^2$ instead of $L^\infty$, obtaining (possibly unbounded) densily defined operators. Then $M_\phi$ and $T_\phi$ are bounded if and only if $\phi\in L^\infty$ (and $\|M_\phi\|=\|T_\phi\|=\|\phi\|_\infty$).
	The situation is more complicated for $H_\phi$. Namely, $H_\phi$ is bounded if and only if there exists $\psi\in L^\infty$ with  $H_\phi= H_{\psi}$, and
	\[
	\|H_\phi\|=\inf \{\|\psi\|_\infty :H_\phi= H_{\psi} \}
	\]
	This  is known as Nehari's Theorem; see, for instance,~\cite[p. 182]{MR1864396}. Moreover (but we will not pursue this in the sequel) an equivalent condition is
	 $P_-\phi\in BMO$ (and $\|H_\phi\|$ is then a norm equivalent to $\|P_-\phi\|_{BMO}$). 
	
	Related results are known for compactness. The operators $M_\phi$ and $T_\phi$ are never compact except in the trivial case $\phi\equiv0$.  Hartman's Theorem states that $H_\phi$ is compact 
	 if and only if there exists $\psi\in \Cont$ with  $H_\phi= H_{\psi}$; or, equivalently,  $P_-\phi\in VMO$. If we know that $\phi$ is bounded, then $H_\phi$ is compact if and only if $\phi\in\Cont+H^\infty$.

	\section{Carleson measures}
	
	\subsection{Embedding of Hardy spaces}
	
	Let us discuss first some objects related to the Hardy space; we will afterwards see what analogous facts are true for the case of model spaces.
	
	A positive measure $\mu$ on $\bbD$ is called a Carleson measure if $H^2\subset L^2(\mu)$ (such an inclusion is automatically continuous). It is known that this is equivalent to $H^p\subset L^p(\mu)$ for all $1\le p<\infty$. Carleson measures can also be characterized by a geometrical condition, as follows. For an arc $I\subset\bbT$ such that $|I|<1$ we define
	\[
	S(I)=\{z\in\bbD: 1-|I|<|z|<1 \mbox{ and } z/|z|\in I\}.
	\]
	Then $\mu$ is a Carleson measure if and only if 
	\begin{equation}\label{eq:carleson}
	\sup_I\frac{\mu(S(I))}{|I|}<\infty.
	\end{equation}
	Condition~\eqref{eq:carleson} is called the \emph{Carleson condition}.
	
	The result can actually be extended (see~\cite{BJ}) to measures defined on $\overline{\bbD}$. Again the characterization does not depend on $p$, and it amounts to the fact that $\mu|_\bbT$ is absolutely continuous with respect to Lebesgue measure with essentially bounded density, while $\mu|_\bbD$ satisfies~\eqref{eq:carleson}.
	
There is a link between Hankel operators and  Carleson measures that has first appeared in~\cite{Pow, Wid}; a comprehensive presentation can be find in~\cite[1.7]{Peller}. Let $\mu$ be a finite complex measure on $\overline{\bbD}$. Define the operator $\Gamma[\mu]$ on analytic polynomials by the formula
\[
\<\Gamma[\mu] f, g\>=\int_{\overline{\bbD}} z f(z)\overline{g(\bar z)}\, d\mu(z).
\]
Note that if $\mu$ is supported on $\bbT$, then the matrix of $\Gamma[\mu]$ in the standard basis of $H^2$ is $(\hat\mu(i+j))_{i,j\ge 0}$, where $\hat{\mu}(i)$ are the Fourier coefficients of~$\mu$.

Then the operator $\Gamma[\mu]$ is bounded whenever $\mu$ is a Carleson measure. Conversely, if $\Gamma[\mu]$ is bounded, then there exists a Carleson measure $\nu$ on $\bbD$ such that $\Gamma[\mu]=\Gamma[\nu]$.

	It is easy to see that if $d\mu=\phi dm$ for some $\phi\in L^\infty$, then $\Gamma[\mu]=\Gamma_\phi$, where $\Gamma_\phi$ has been defined by~\eqref{eq:def of Gamma} and is the version of a Hankel operator acting on a single space.

	Analogous results may be proved concerning compactness. In this case the relevant notion is that of \emph{vanishing Carleson measure}, which is defined by the property
		\begin{equation}\label{eq:vcarleson}
		\lim_{|I|\to0}\frac{\mu(S(I))}{|I|}=0.
		\end{equation}
		Note that vanishing Carleson measures cannot have mass on the unit circle (intervals containing a Lebesgue point of the corresponding density would contradict the vanishing condition).
		Then the embedding $H^p\subset L^p(\mu)$ is  compact if and only if $\mu$ is a vanishing Carleson measure.

	A similar connection exists to compactness of Hankel operators. If $\mu$ is a vanishing Carleson measure on $\overline{\bbD}$, then  $\Gamma[\mu]$ is compact. Conversely, if $\Gamma[\mu]$ is compact, then there exists a vanishing Carleson measure $\nu$ on $\overline{\bbD}$ such that $\Gamma[\mu]=\Gamma[\nu]$.

		\subsection{Embedding of model spaces}\label{subsection-carl-model-space}

	 Similar questions for model spaces have been developed starting with the papers~\cite{Cohn1,Cohn2} and~\cite{VT}; however, the results in this case are less complete. Let us introduce first some notations. For $1\le p <\infty$, define
	 \[
	 \begin{split}
	 \CC_p(\Theta)&=\{\mu \text{ finite measure on }\bbT: K^p_\Theta\hookrightarrow L^p(|\mu|) \text{ is bounded}\},\\
	 \CC^+_p(\Theta)&=\{\mu \text{ positive measure on }\bbT: K^p_\Theta\hookrightarrow L^p(\mu) \text{ is bounded}\},\\
	 \VV_p(\Theta)&=\{\mu \text{ finite measure on }\bbT: K^p_\Theta\hookrightarrow L^p(|\mu|) \text{ is compact}\},\\
	 \VV^+_p(\Theta)&=\{\mu \text{ positive measure on }\bbT: K^p_\Theta\hookrightarrow L^p(\mu) \text{ is compact}\}.
	 \end{split} 
	 \]
	 
	 It is clear that $\CC_p(\Theta)$ and $\VV_p(\Theta)$ are complex vectorial subspaces of the complex measures on the unit circle. Using the  relations $K_{\Theta^2}=K_\Theta\oplus \Theta K_\Theta$ and $K_\Theta \cdot K_\Theta\subset K_{\Theta^2}^1$, it is easy to see that  $\CC_2(\Theta^2)=\CC_2(\Theta)$, $\CC_1(\Theta^2)\subset\CC_2(\Theta)$, and $\VV_1(\Theta^2)\subset \VV_2(\Theta)$.

	 It is natural to look for geometric conditions to characterize these classes. Things are, however, more complicated, and the results are only partial. We start by fixing a number $0<\epsilon<1$; then the $(\Theta,\epsilon)$-\emph{Carleson condition} asserts that 
	 \begin{equation}\label{eq:thetacarleson}
	 \sup_I\frac{\mu(S(I))}{|I|}<\infty,
	 \end{equation}
	 where the supremum is taken only over the intervals $|I|$ such that $S(I)\cap \Omega(\Theta,\epsilon)\not=\emptyset $. (Remember that $\Omega(\Theta,\epsilon)$ is given by~\eqref{eq:definition of Omega}.)
	 
	 It is then proved in~\cite{VT} that if $\mu$ satisfies the $(\Theta,\epsilon)$-Carleson condition, then the embedding $K_\Theta^p\subset L^p(\mu)$ is continuous. The converse is true if $\Theta$ is \emph{one-component}; in which case  the embedding condition does not depend on $p$, while fulfilling of the $(\Theta, \epsilon)$-Carleson condition does not depend on $0<\epsilon<1$ (see Theorem~\ref{th:Aleksandrov one component} below).

As concerns the general case, it is shown by Aleksandrov~\cite{A2} that if the converse is true for some $1\le p <\infty$, then $\Theta$ is one-component. Also, $\Theta$ is one-component if and only if the embedding condition does not depend on~$p$. More precisely, the next theorem is proved in~\cite{A2} (note that a version of this result for $p\in (1,\infty)$ already appears in \cite{VT}).

\begin{theorem}\label{th:Aleksandrov one component}
	The following are equivalent for an inner function $\Theta$:
	\begin{enumerate}
		\item $\Theta$ is one-component.
		\item For some $0<p<\infty$ and $0<\epsilon<1$, $\CC_p(\Theta)$ concides with the class of measures that satisfy the $(\Theta, \epsilon)$-Carleson condition.
		\item For all $0<p<\infty$ and $0<\epsilon<1$, $\CC_p(\Theta)$ concides with the class of measures that satisfy the $(\Theta, \epsilon)$-Carleson condition.
		\item The class $\CC_p(\Theta)$ does not depend on $p\in(0, \infty)$.
	\end{enumerate}
\end{theorem}

 In particular, if $\Theta$ is one component, then so is $\Theta^2$, whence $\CC_1(\Theta^2)=\CC_2(\Theta^2)=\CC_2(\Theta)$.

Note that a general characterization of $\CC_2(\Theta)$ has recently been obtained in~\cite{LW}; however, the  geometric content of this result is not easy to see.

The question of compactness of the embedding $K^p_\Theta\subset L^p(\mu)$ in this case should be related to a vanishing Carleson condition. In fact, there are two vanishing conditions, introduced in~\cite{CM}. What is called therein the \emph{second vanishing condition} is easier to state.
We say that $\mu$ satisfies the \emph{second $(\Theta, \epsilon)$-vanishing condition}~\cite{Ba, CM} if for each $\eta>0$ there exists $\delta>0$ such that $\mu(S(I))/|I|<\eta$ whenever $|I|<\delta$ and $S(I)\cap \Omega(\Theta, \epsilon)\not=\emptyset$. 
The following result is then proved in~\cite{Ba}.

\begin{theorem}\label{th:baranov compact embeddings}
	If the positive measure $\mu$ satisfies the second $(\Theta, \epsilon)$-vanishing condition, then the embedding $K^p_\Theta\subset L^p(\mu)$ is compact for $1<p<\infty$.
	
	The converse is true in case $\Theta$ is one-component.
\end{theorem}

In other words, the theorem thus states that positive measures that satisfy the second vanishing condition are in $\VV_p^+(\Theta)$ for all $1<p<\infty$, and the converse is true for $\Theta$ one-component.

To discuss the case $p=1$, we have to introduce what is called in~\cite{CM}  the \emph{first vanishing condition}.
Let us  call the supremum in~\eqref{eq:thetacarleson}
the $(\Theta, \epsilon)$-Carleson constant of $\mu$. 
Define 
\begin{equation}\label{eq:definition of H_delta}
H_\delta= \{z\in \overline{\bbD}: \dist(z,\rho(\Theta))<\delta\},
\end{equation}
and $\mu_\delta(A)=\mu(A\cap H_\delta)$. Then $\mu_\delta$ are also $\Theta$-Carleson measures, with $(\Theta, \epsilon)$-Carleson constants decreasing when $\delta$ decreases. We say that $\mu$ \emph{satisfies the first $(\Theta, \epsilon)$-vanishing condition} if these Carleson constants tend to~0 when $\delta\to 0$.

It is shown in~\cite{Ba} that the first vanishing condition implies the second, and that the converse is not true: there exist measures which satisfy the second vanishing condition but not the first.

 The next theorem is proved in~\cite{CM}.
 
 \begin{theorem}\label{th:cima-matheson}
 	If a positive measure  $\mu$ satisfies the first $(\Theta, \epsilon)$-vanishing condition, then $\mu\in \VV^+_p(\Theta)$ for $1\leq p<\infty$.
 \end{theorem}

 In case $\mu\in \CC_p(\Theta)$, we will denote by $\iota_{\mu,p}:K^p_\Theta\to L^p(|\mu|)$ the embedding (which is then known to be a bounded operator). Then $\mu\in \VV_p(\Theta)$ means that $\iota_{\mu,p}$ is compact. We will also write $\iota_\mu$ instead of $\iota_{\mu,2}$.

\section{Truncated Toeplitz operators}	\label{se:TTO} 
	 
	 Let $\Theta$ be an inner function and $\phi\in L^2$. The truncated Toeplitz operator $A_\phi=A_\phi^\Theta$, introduced by Sarason in \cite{Sa}, will be a densely defined, possibly unbounded operator on $K_\Theta$. Its domain is $K_\Theta\cap H^\infty$, on which it acts by the formula
	 $$
	 A_\phi f=P_\Theta(\phi f),\qquad f\in K_\Theta\cap H^\infty.
	 $$
	 If $A_\phi$ thus defined extends to a bounded operator, that operator is called a TTO. The class of all TTOs on $K_\Theta$ is denoted by $\TT(\Theta)$, and the class of all nonnegative TTO's on $K_\Theta$ is denoted by $\TT(\Theta)^+$. 

Although these operators are called truncated \emph{Toeplitz}, they have more in common with \emph{Hankel} operators $H_\phi$, or rather with their cognates $\Gamma_\phi$, which act on a single space. 
	 As a first example of this behavior, we note that the symbol of a truncated Toeplitz operators is not unique.  It is proved in \cite{Sa} that 
	 \begin{equation}\label{eq:symbols zero}
	 	 A_{\phi_1}=A_{\phi_2}\Longleftrightarrow \phi_1-\phi_2\in \Theta H^2+\overline{\Theta H^2}.
	 \end{equation}
	 Let us denote $\mathfrak{S}_\Theta=L^2\ominus (\Theta H^2+\overline{\Theta H^2})$; it is called the space of \emph{standard} symbols. It follows from~\eqref{eq:symbols zero} that every TTO has a unique standard symbol. One proves in~\cite[Section 3]{Sa} that $\mathfrak{S}$ is contained in $K_\Theta+\overline{K_\Theta}$ as a subspace of codimension at most one; this last space is sometimes easier to work with.

	 	It is often the case that the assumption $\Theta(0)=0$ simplifies certain calculations. For instance, in that case we have precisely $\mathfrak{S}=K_\Theta+\overline{K_\Theta}$; we will see another example in Section~\ref{se:TTOs in ideals}. Fortunately, there is a procedure to pass from a general inner $\Theta$ to one that has this property: it is called the \emph{Crofoot transform}. For $a\in\bbD$ let $\Theta_a$ be given by the formula
	 	\[
	 	\Theta_a(z)= \frac{\Theta(z)-a}{1-\bar a\Theta(z)}.
	 	\]
	 	If we define the Crofoot transform by 
	 	\[
	 	J(f):= \frac{\sqrt{1-|a|^2}}{1-\bar a\Theta}f,
	 	\]
	 	then $J$ is a unitary operator from $K_\Theta$ to $K_{\Theta_a}$, and
	 	\begin{equation}\label{eq:crofoot}
	 	J\TT(\Theta)J^*=\TT(\Theta_a).
	 	\end{equation}
	 	  In particular, if $a=\Theta(0)$, then $\Theta_a(0)=0$, and~\eqref{eq:crofoot} allows the transfer of properties from TTOs on $K_{\Theta_a}$ to TTOs on $K_\Theta$.

Especially nice properties are exhibited by TTOs which have an analytic symbol $\phi\in H^2$ (of course, this is never a standard symbol). It is a consequence of interpolation theory~\cite{SaGint} that 
\[
\{A^\Theta_\phi\in \TT(\Theta): \phi\in H^2 \}= \{A^\Theta_z\}'
\] 
($A^\Theta_z$ is called a \emph{compressed shift}, or a \emph{model operator}).	 

One should also mentioned that other two classes of TTOs  have already been studied in different contexts. First, the classical finite Toeplitz matrices are TTOs with $\Theta(z)=z^n$ written in the  basis of monomials. Secondly, TTOs with $\Theta(z)=e^{\frac{z+1}{z-1}}$ correspond, after some standard transformations,  to a class of operators alternately called Toeplitz operators on Paley--Wiener spaces~\cite{RR}, or truncated Wiener--Hopf operators~\cite{BGS}. 
	 
	 There is an alternate manner to introduce TTOs, related to the Carleson measures in the previous section. For every $\mu\in \CC_2(\Theta)$ the sesquilinear form
	 \[
	 (f, g)\mapsto \int f \bar g \, d\mu
	 \]
	 is bounded, and therefore there exists a bounded operator $A^\Theta_\mu$ on $K_\Theta$ such that
	 \begin{equation}\label{eq:tto by measure}
	 	  \<A^\Theta_\mu f, g\>= \int f \bar g \, d\mu.
	 \end{equation}
	 
	It is shown in~\cite[Theorem 9.1]{Sa} that $A^\Theta_\mu$ thus defined is actually a TTO. In fact, the converse is also true, as stated in Theorem~\ref{th:main theorem BBK} below.  An interesting open question is the characterization of the measures $\mu$ for which $A_\mu=0$.

The definition of TTOs does not make precise the class of symbols  $\phi\in L^2$ that produce  bounded TTOs. 
A first remark is that the standard symbol of a bounded truncated Toeplitz operator is not necessarily bounded. To give an example, consider an inner function $\Theta$ with $\Theta(0)=0$, for which there exists a singular point $\zeta\in\bbT$ where $\Theta$ has an angular derivative in the sense of Caratheodory. It is shown then in~\cite[Section 5]{Sa} that $k_\zeta^\Theta\otimes k_\zeta^\Theta $ is a bounded rank one TTO with standard symbol $k_\zeta^\Theta+ \overline{k_\zeta^\Theta}-1$, and that this last function is unbounded.

A natural question is therefore whether every bounded TTO has a bounded symbol (such as is the case with Hankel operators). In the case of $T_\phi$ with $\phi$ analytic, the answer is readily seen to be positive, being proved again in~\cite{SaGint}; moreover, 
\[
\inf\{\|\psi\|_\infty: \psi\in H^\infty, \ A^\Theta_\psi=A^\Theta_\phi \}=\|A^\Theta_\phi\|.
\]

The first  negative answer for the general situation has been  provided in~\cite{MR2679022}, and the counterexample is again given by the rank one TTO $k_\zeta^\Theta\otimes k_\zeta^\Theta $.  The following result is  proved in~\cite{MR2679022}.

\begin{theorem}\label{th:counterex BCFMT}
	Suppose $\Theta$ has an angular derivative in the sense of Caratheodory in $\zeta\in\bbT$ (equivalently, $k^\Theta_\zeta\in L^2$), but $k^\Theta_\zeta\not\in L^p$ for some $p\in(2,\infty)$. Then $k^\Theta_\zeta\otimes k^\Theta_\zeta$ has no bounded symbol.
\end{theorem}

A more general result has been obtain in~\cite{BBK}, where one also makes clear the relation between measures and TTO. In particular, one characterizes the inner functions $\Theta$ which have the property that every bounded TTO on $K_\Theta$ has a bounded symbol.

\begin{theorem}\label{th:main theorem BBK}
Suppose $\Theta$ is an inner function.
\begin{enumerate}
	\item For every  bounded TTO $A\ge 0$ there exists a positive measure $\mu\in \CC^+_2(\Theta)$ such that $A=A^\Theta_\mu$.
	
	\item For every bounded  $A\in\TT(\Theta)$ there exists a complex measure $\mu\in \CC_2(\Theta)$ such that $A=A^\Theta_\mu$.
	
	\item A bounded TTO $A\in\TT(\Theta)$ admits a bounded symbol if and only if $A=A^\Theta_\mu$ for some $\mu\in \CC_1(\Theta^2)$.
	
	\item Every bounded TTO on $K_\Theta$ admits a bounded symbol if and only if $\CC_1(\Theta^2)=\CC_2(\Theta^2)$.
\end{enumerate}	
\end{theorem}

In particular, as shown by Theorem~\ref{th:Aleksandrov one component}, the second condition is satisfied if $\Theta$ is one-component (since then all classes $\CC_p(\Theta)$ coincide). It is still an open question whether $\Theta$ one-component is actually equivalent to $\CC_1(\Theta^2)=\CC_2(\Theta^2)$. (As mentioned previously, $\Theta$ is one-component if and only if $\Theta^2$ is one-component.) Such a result would be a significant strengthening of Theorem~\ref{th:Aleksandrov one component}.

As a general observation, one may say that, if $\Theta$ is one-component, then TTOs on $K_\Theta$ have many properties analogous to those of Hankel operators. This is the class of inner functions for which the current theory is more developed.

\section{Compact operators}

Surprisingly enough, the first result about compactness of TTOs dates from 1970. In~\cite[Section 5]{AC} one introduces what are, in our terminology, TTOs with continuous symbol, and one proves the following theorem.

\begin{theorem}\label{th:tto with continuous symbols}
	If $\Theta$ is inner and $\phi$ is continuous on $\bbT$, then  $A^\Theta_\phi$ is compact if and only if $\phi|\rho(\Theta)=0$.
\end{theorem}
This result has been rediscovered more recently in~\cite{MR3203060}; see also ~\cite{GR}.

Thinking of Hartmann's theorem, it seems plausible to believe that continuous symbols play for compact TTOs the role played by bounded symbols for general TTOs. However,  as shown by Theorem~\ref{th:counterex BCFMT}, there exist inner functions $\Theta$ for which even rank-one operators might not have bounded symbols (not to speak about continuous). So we have to consider only certain classes of inner functions, suggested by the boundedness results in the previous section. In this sense one has the following result proved by Bessonov~\cite{Be}.

\begin{theorem}\label{co:besonov main-cor for TTO} 
	Suppose that $\CC_1(\Theta^2)=\CC_2(\Theta)$ and $A$ is a truncated truncated Toeplitz operator. Then the following are equivalent:
	\begin{enumerate}
		\item $A$ is compact.
		\item $A=A_{\phi\Theta}$ for some $\phi\in\Cont$.
	\end{enumerate}
	
\end{theorem}
 In particular, this is true if $\Theta$ is one component. 
 
One can see that instead of $\Cont$ the main role is played by $\Theta\Cont$. We give below some ideas about the connection between these two classes.

 \begin{theorem}\label{th:spectrum of measure zero}
 	Suppose $\CC_2(\Theta)=\CC_1(\Theta^2)$ and $m(\rho(\Theta))=0$. Then the following are equivalent for a truncated Toeplitz operator $A$.
 	\begin{itemize}
 		\item[(i)] $A$ is compact.
 		\item[(ii)] $A=A^\Theta_\phi$ for some $\phi\in\Cont$ with $\phi|\rho(\Theta)=0$.
 	\end{itemize}
 \end{theorem} 
 
 \begin{proof}
 	(ii)$\implies$(i) is proved in Theorem~\ref{th:tto with continuous symbols}.
 	
 	Suppose now (i) is true. By Theorem~\ref{co:besonov main-cor for TTO} $A=A_{\Theta\psi}$ for some $\psi\in\Cont$.
 	By the Rudin--Carleson interpolation theorem (see, for instance, \cite[Theorem II.12.6]{Ga}), there exists a function $\psi_1\in \Cont\cap H^\infty$ (that is, in the disk algebra) such that $\psi|\rho(\Theta)= \psi_1|\rho(\Theta)$. Then one checks easily that $\phi=\Theta(\psi-\psi_1)$ is continuous on $\bbT$, and $A_\phi=A_{\Theta\psi}$ (since $A_{\Theta\psi_1}=0$).	
 \end{proof} 
 In particular, Theorem~\ref{th:spectrum of measure zero} applies to the case $\Theta$ one-component, since for such functions we have $\CC_2(\Theta)=\CC_1(\Theta^2)$  and  
 $m(\rho(\Theta))=0$~\cite[Theorem 6.4]{A1}.
 
 We also have the following result which is contained in \cite[Proposition 2.1]{Be}; here is a simpler proof.
 
 \begin{proposition}\label{pr:besonov simple}
 	\begin{itemize}
 		\item[(i)]
 		If $\phi\in \Theta\Cont +\Theta H^\infty$, then $A_\phi$ is compact.
 		\item[(ii)]
 		If $\phi\in \Cont + H^\infty$, then the converse is also true.
 	\end{itemize}
 \end{proposition} 
 
 \begin{proof}
 	First note that 
 	\begin{align}
 	A_\phi&=(\Theta H_{\bar\Theta \phi}-H_\phi)|K_\Theta. \label{eq:relation TTO Hankel}
 	\end{align}
 	By Hartmann's Theorem we know that a Hankel operator with bounded symbol is compact if and only if its symbol is in $\Cont + H^\infty$.  Since $\Cont + H^\infty$ is an algebra,  $\phi\in \Theta\Cont +\Theta H^\infty$, that is, $\bar\Theta  \phi\in \Cont + H^\infty$, implies  $ \phi\in \Cont + H^\infty$. Then applying~\eqref{eq:relation TTO Hankel} proves (i).
 	
 	On the other hand, if $\phi\in \Cont + H^\infty$,  again~\eqref{eq:relation TTO Hankel} proves (ii).
 \end{proof}
 
It is interesting to compare Theorem~\ref{co:besonov main-cor for TTO} to Proposition~\ref{pr:besonov simple}.  Suppose  that a TTO $A_\phi$ is compact. Proposition~\ref{pr:besonov simple} says that, if we know that $\phi\in \Cont + H^\infty$, then it has actually to be in $\Theta\Cont +\Theta H^\infty$. So there exists $\psi \in \Cont + H^\infty$ such that $\phi=\Theta \psi$. This is true with no special assumption on $\Theta$, but the symbol $\phi$ is assumed to be in a particular class.

On the other hand, suppose that $\Theta$ satisfies the assumption $\CC_2(\Theta)=\CC_1(\Theta^2)$, and again $A_\phi$ is compact. Without any a priori assumption on the symbol, applying Theorem~\ref{co:besonov main-cor for TTO} yields the existence of $\psi\in \Cont + H^\infty$  such that $A_\phi=A_{\Theta \psi}$. However, in this case we will not necessarily have $\phi=\Theta \psi$, but, according to~\eqref{eq:symbols zero}, $\phi-\Theta\psi\in \Theta H^2+\overline{\Theta H^2}$.

 	It would be interesting to give an example of a compact operator, with a symbol  $\psi\in\Theta \Cont+\Theta H^\infty$, that has no continuous symbol.

 Since $A_\phi$ is compact if and only if $A_\phi^*=A_{\bar{\phi}}$ is, any condition on the symbol produces another one by  conjugation. So one expects a definitive result to be invariant by conjugation. This is not the case, for instance, with Proposition~\ref{pr:besonov simple}: by conjugation we obtain that if $\phi\in \bar{\Theta}\Cont +\overline{\Theta H^\infty}$, then $A^\Theta_\phi$ is compact. Also, in Theorem~\ref{th:tto with continuous symbols} one could add a third equivalent condition, namely that $A=A_{\phi\overline{\Theta}}$ for some $\phi\in\Cont$. From this point of view, Theorem~\ref{th:spectrum of measure zero} is more satisfactory. Naturally, if $\Theta$ is one-component would actually be equivalent to $\CC_1(\Theta^2)=\CC_2(\Theta)$ (the open question stated above), then  Theorems~\ref{th:tto with continuous symbols} and~\ref{th:spectrum of measure zero} would turn out to be equivalent to a simple and symmetric statement for this class of functions.

\section{Compact TTOs and embedding measures}

In the present section we discuss some relations between compactness of TTOs and embedding measures. Let us first remember  that a (finite) complex measure on the unit circle can be decomposed by means of nonnegative finite measures, as stated more precisely in the following lemma \cite[chap. 6]{rudin}.    
\begin{lemma}\label{le:decomposition of measures}
	If $\mu$ is a complex measure, one can write $\mu=\mu_1-\mu_2+i \mu_3-i\mu_4$ with $0\le \mu_j\le |\mu|$ for $1\le j\le 4$.
\end{lemma}

We will also use the following simple result.

\begin{lemma}\label{le:domination of measures}
	If $0\le \nu_1\le \nu_2$, then $\nu_2\in \CC^+_p(\Theta) $ implies $\nu_1\in \CC^+_p(\Theta) $, and $\nu_2\in \VV^+_p(\Theta) $ implies $\nu_1\in \VV^+_p(\Theta) $.
\end{lemma}
\begin{proof}
	If $0\le \nu_1\le \nu_2$, then we have a contractive embedding $J:L^p(\nu_2)\to L^p(\nu_1)$, and the lemma follows from the equality $\iota_{\nu_1,p}=J\iota_{\nu_2,p}$.
\end{proof}

The ultimate goal would be  to obtain for compact TTOs statements similar to those for boundedness appearing in Theorem~\ref{th:main theorem BBK}. But one can only obtain partial results: measures in $\VV_2(\Theta)$ produce compact TTOs, but the converse can be obtained only for positive operators.

\begin{theorem}\label{th:compact corresponds to V(theta)}
Suppose $A\in\TT(\Theta)$.
\begin{enumerate}
	\item
	If there exists $\mu\in \VV_2(\Theta)$ such that $A=A_\mu$, then $A$ is compact.
	\item
	If $A$ is compact and positive, then there exists $\mu\in \nu_2^+(\Theta)$ such that $A=A_\mu$.
\end{enumerate}
	
\end{theorem}

\begin{proof}
	(1)
	Take $A_\mu$ with $\mu\in \VV_2(\Theta)$. Writing $\mu=\mu_1-\mu_2+i \mu_3-i\mu_4$ as in Lemma~\ref{le:decomposition of measures}, one has  $A_\mu=A_{\mu_1}- A_{\mu_2}+iA_{\mu_3}-iA_{\mu_4}$. Since $0\leq \mu_j\leq |\mu|$, it follows that $\mu_j\in \VV_2^+(\Theta)$ by Lemma~\ref{le:domination of measures}. So we may suppose from the beginning that $\mu \in\VV_2^+(\Theta) $.
	
	To show that $A_\mu$ is compact, take a sequence $(f_n)$ tending weakly to 0 in $K_\Theta$, and $g\in K_\Theta$ with $\|g\|_2=1$. Formula~\eqref{eq:tto by measure} can be written
	\[
	\<A_\mu f_n, g\>_2=\int \iota_\mu (f_n) \overline{\iota_\mu ( g)} d\mu,
	\]
	and thus
	\[
	|\<A_\mu f_n, g\>|\le \|\iota_\mu (f_n)\|_{L^2(\mu)} \|\iota_\mu (g)\| _{L^2(\mu)} \le \|\iota_\mu (f_n)\|_{L^2(\mu)} \|\iota_\mu\|.
	\]
	Taking the supremum with respect to $g$, we obtain
	\[
	\|A_\mu f_n\|_2\le \|\iota_\mu (f_n)\|_{L^2(\mu)}  \|\iota_\mu\|.
	\]
	But $f_n\to 0$ weakly and $\iota_\mu$ compact imply that $\|\iota_\mu (f_n)\|_{L^2(\mu)} \to 0$. So $\|A_\mu f_n\|\to 0$ and therefore $A_\mu$ is compact.
	
\smallskip (2)
	If $A\geq 0$, by Theorem~\ref{th:main theorem BBK}, there exists $\mu\in \CC_2^+(\Theta)$ such that $A=A_\mu$. We must then show that $\mu\in \VV_2^+(\Theta)$; that is,  $\iota_\mu$ is compact.
	
	Take then a sequence $f_n$ tending weakly to 0 in $K_\Theta$; in particular, $(f_n)$ is bounded, so we may assume $\|f_n\|\le M$ for all~$n$. Applying again formula~\eqref{eq:tto by measure}, we have
	\[
	\<A_\mu f_n, f_n\>=\int \iota_\mu (f_n) \overline{\iota_\mu (f_n)} d\mu = \|\iota_\mu (f_n)\|_{L^2(\mu)}^2.
	\]
	Therefore
	\[
	\|\iota_\mu (f_n)\|_{L^2(\mu)}^2\le \|A_\mu f_n\| \|f_n\|\le M\|A_\mu f_n\|. 
	\]
	Since $A_\mu$ is compact, $ \|A_\mu f_n\|\to 0$. The same is then true about $\|\iota_\mu (f_n)\|_{L^2(\mu)}^2$; thus $\iota_\mu$ is compact, that is, $\mu\in \VV_2^+(\Theta)$.
	\end{proof}	
%
%
%

This approach leads to an alternate proof of Theorem~\ref{th:tto with continuous symbols}.

\begin{proposition}\label{pr:one direction for cont symbols}
	If $\phi\in  \Cont$ and $\phi|\rho(\Theta)=0$, then the measure $\mu=\phi dm$ is in $\VV_p(\Theta)$ for every $1\leq p<\infty$. In particular, $A_\phi$ is compact.
\end{proposition}

\begin{proof}
	Since $\phi\in L^\infty$, the measure $|\mu|$ is an obvious $\Theta$-Carleson measure. Now fix $\epsilon>0$. Since $\phi$ is uniformly continous on $\bbT$, there exists $\eta>0$ such that, if $\zeta\in\bbT$, $\dist(\zeta, \rho(\Theta))<\eta$, then $|\phi(\zeta)|<\epsilon$. In other words, if $\zeta\in H_\eta$, then $|\phi(\zeta)|<\epsilon$ (where $H_\eta$ is defined by~\eqref{eq:definition of H_delta}). 
	
	 Let $\delta<\eta$ and $I$ be any arc of $\bbT$. Then we have
	\[
	\begin{split}
		|\mu|_\delta(T(I))&=|\mu|(T(I)\cap H_\delta)\\
		&=\sup
		\left\{\sum_{i\geq 1}|\mu(E_i)|: \bigcup_{i\geq 1} E_i=T(I)\cap H_\delta, \ E_i\cap E_j=\emptyset\text{ for }i\not=j\right\}
	\end{split}
	\]
	Since $E_i\subset H_\delta\subset H_\eta$, note that
	$$
	\begin{aligned}
	|\mu(E_i)|&=\left|\int_{E_i}\phi\,dm\right|\\
	&\leq \int_{E_i}|\phi|\,dm\leq \epsilon \,m(E_i).
	\end{aligned}
	$$ 
	Hence
	$$
	|\mu|_\delta(T(I))\leq \epsilon |I|,
	$$
	which shows that 
	the $\Theta$-Carleson constant of $|\mu|_\delta$ is smaller than $\epsilon$. We conclude the proof applying Theorem~\ref{th:cima-matheson} and Theorem~\ref{th:compact corresponds to V(theta)}.
\end{proof}

The next theorem is a partial analogue of Theorem~\ref{th:main theorem BBK} (3).

\begin{theorem}\label{th:V_1 and continuous symbols}
	Suppose $\mu\in\VV_1(\Theta^2)$. Then $A_\mu=A_{\Theta\phi}$ for some $\phi\in\Cont$.
\end{theorem}

\begin{proof} By~\cite[Lemma 3.1]{Be} we know that $K^1_{z\Theta^2}\cap zH^1$ is w*-closed when we consider it embedded in $H^1=\Cont/H^1_0$.
	We define on $K^1_{z\Theta^2}\cap zH^1$ the linear functional $\ell$ by
	\[
	\ell(f)=\int \bar{\Theta} f\, d\mu.
	\]
	It is clear that $\ell$ is continuous, but we assert that it is also w*-continuous. Indeed, the w* topology is metrizable (since $\Cont$ is separable), and therefore we can check w*-continuity on sequences. 
	
	If $f_n\to0$ w*, then, in particular, the sequence $(f_n)$ is bounded. Then, since $\iota_\mu$ is compact, the sequence $(\iota_\mu (f_n))$ is compact in $L^1(\mu)$, and a standard argument says that, in fact, $\iota_\mu f_n\to 0$ in  $L^1(\mu)$. Then 
	\[
	\ell(f_n)=\int \bar{\Theta}(\iota_\mu f_n)\, d\mu\to 0.
	\]
	
	It follows that there exists $\phi\in\Cont$, such that 
	\[
	\ell(f)=\int \bar{\Theta} f\, d\mu=\int \phi f\, dm
	\]
	for every $f\in K^1_{z\Theta^2}\cap zH^1$, or, equivalently,
	\[
	\int  f\, d\mu=\int\Theta \phi f\, dm
	\]
	for every $f\in\bar{\Theta}(K^1_{z\Theta^2}\cap zH^1)$. If $g,h\in K^2_\Theta$, then $g\bar h\in \bar{\Theta}(K^1_{z\Theta^2}\cap zH^1)$ , so
	\[
	\< A_\mu g, h\>= \int g\bar h\, d\mu=
	\int \Theta\phi g\bar h\, dm =\<A_{\Theta\phi}f, g\>,
	\]
	which proves the theorem.
\end{proof}

In particular, it follows from Proposition~\ref{pr:besonov simple} that  if $\mu\in\VV_1(\Theta^2)$ then $A^\Theta_\mu$ is compact.

\section{TTOs in other ideals}\label{se:TTOs in ideals}

The problem of deciding when certain TTOs are in Schatten--von Neumann classes $S_p$ has no clear solution yet, even in the usually simple case of the Hilbert--Schmidt ideal. In~\cite{LR} one gives criteria for particular cases; to convey their flavour, below is an example (Theorem~3 of \cite{LR}). Remember that $\Theta$ is called an interpolating Blaschke product if its zeros $(z_i)$  form an interpolation sequence, or, equivalently, if they satisfy the Carleson condition
\[
\inf_{i\in\bbN} \prod_{j\not=i}\left| \frac{z_i-z_j}{1-\bar z_i z_j}\right|>0.
\]

\begin{theorem}\label{th:lopatto-rochberg l_p}
	Suppose $\Theta$ is an interpolating Blaschke product and $\phi$ is an analytic function. Then:
	\begin{enumerate}
		\item $A_\phi$ is compact if and only if $\phi(z_i)\to 0$.
		\item For $1\le p<\infty$, $A_\phi\in S_p$ if and only if $(\phi(z_i))\in \ell^p$.
	\end{enumerate}  
\end{theorem}

More satisfactory results are obtained in~\cite{LR} in the case of Hilbert--Schmidt operators, but even in this case an explicit equivalent condition on the symbol is hard to formulate. Let us start by assuming that $\Theta(0)=0$ (see the discussion of the Crofoot transform in Section~\ref{se:TTO}); in this case the space of standard symbols $\mathfrak{S}$ is precisely $K_\Theta+\overline{K_\Theta}$. We define then  $\Wth=\Theta^2/z$; $\Wth$ is also an inner function with $\Wth(0)=0$, and   $C_\Theta(K_\Theta+\overline{K_\Theta})=K_\Wth$ (remember that $C_\Theta$ is given by formula~\eqref{eq:conjugation}).

Let then $K_\Wth^0$ be the linear span (nonclosed) of the reproducing kernels $k^\Wth_\lambda$, $\lambda\in\bbD$. It can be checked that for every $\lambda\in\bbD$ we have $(k^\Theta_\lambda)^2\in K_\Wth$, and therefore the formula
\[
\DD_0 k^{\Phi}_\lambda=(k^\Theta_\lambda)^2
\]
defines an (unbounded) linear operator $\DD_0:K_\Wth^0\to K_\Wth$.

The result that is proved in~\cite{LR} is the following.

\begin{theorem}\label{th:rochberg HS}
With the above notations, the following assertions are true:
	\begin{enumerate}
		\item $\DD_0$ is a positive symmetric operator.  Its Friedrichs selfadjoint extension (see~\cite[Theorem X.23]{RS})  will be denoted by $\DD$; it has a positive root $\DD^{1/2}$.
		\item A TTO $A^\Theta_\phi$, with $\phi\in K_\Theta+\overline{K_\Theta}$, is a Hilbert--Schmidt operator if and only if $C_\Theta \phi$ is in the domain of $\DD^{1/2}$, and the Hilbert--Schmidt norm is $\|\DD^{1/2} (C_\Theta \phi)\|$. 
	\end{enumerate}
\end{theorem}

Since the square of a reproducing kernel is also a reproducing kernel, let us denote by $\HH^2_\Theta$ the reproducing kernel Hilbert space that has as kernels $(k^\Theta_\lambda)^2$ ($\lambda\in\bbD$). It is a space of analytic functions defined on $\bbD$, and it provides another characterization of Hilbert--Schmidt TTOs obtained in~\cite{LR}.

\begin{theorem}\label{th:rochberg HS2}
	Define, for $\phi\in K_\Theta+\overline{K_\Theta} $, 
	\begin{equation*}\label{eq:definition of DD}
	(\Delta \phi)(\lambda)=\<C_\Theta \phi, (k^\Theta_\lambda)^2\>.
	\end{equation*}
	Then:
	\begin{enumerate}
		\item $\Delta\phi$ is a function analytic in $\bbD$, which coincides on $K^0_\Phi$ with $\DD(C_\Theta\phi)$.
		\item An alternate formula for $\Delta\phi$ is 
		\begin{equation*}\label{eq:great formula for DD}
		(\Delta\phi)(\lambda)= (z\alpha)'(\lambda)-2\Theta(\lambda)(z\alpha_2)'(\lambda),
		\end{equation*}
		where $C_\Theta\phi=\alpha=\alpha_1+\Theta\alpha_2$, with $\alpha_1, \alpha_2\in K_\Theta$.
		\item $A^\Theta_\phi$ is a Hilbert--Schmidt operator if and only if $\Delta\phi\in \HH^2_\Theta$, and the Hilbert--Schmidt norm is $\|\Delta\phi\|_{\HH^2_\Theta}$.
	\end{enumerate}
\end{theorem}

The proof of these two theorems uses 
the theory of Hankel forms as developed in~\cite{PR}. Admittedly, none of the characterizations is very explicit.

For the case of one-component functions, a conjecture is proposed in~\cite[4.3]{Be} for the characterization of Schatten--von Neumann TTOs. It states essentially that a truncated Toeplitz operator is in $S_p$ if and only if it has at least one symbol $\phi$ in the Besov space $B_{pp}^{1/p}$ (note that this would not necessarily be the standard symbol). This last space admits several equivalent characterizations; for instance, if we define, for $\tau\in\bbT$, $\Delta_\tau f(z)=f(\tau z)-f(z)$, then 
\[
B_{pp}^{1/p}=\left\{f\in L^p: \int_\bbT  \frac{\|\Delta_\tau f\|^p_p}{|1-\tau|^2}\, dm(\tau)<\infty \right\}.
\]
The conjecture is suggested by the similar result in the case of Hankel operators~\cite[Chapter 6]{Peller}. It is true if $\Theta(z)=e^{\frac{z+1}{z-1}}$, as shown in~\cite{RR}.

Bessonov also proposes some alternate characterizations in terms of Clark measures; we will not pursue this approach here.

\section{Invertibility and Fredholmness}

Invertibility and, more generally, spectrum of a TTO has been known since several decades in the case of analytic symbols. The main result here is stated in the next theorem (see, for instance, ~\cite[2.5.7]{NikOfs2}). It essentially says that $\sigma(A_\phi^\Theta)=\phi(\Sp(\Theta))$, but we have to give a precise meaning to the quantity on the right, since $\Sp(\Theta)$ (as defined by~\eqref{eq:def spectrum of theta}) intersects the set $\bbT$, where $\phi\in H^\infty$ is defined only almost everywhere.

\begin{theorem}\label{th:spectral mapping theorem}
	If $\phi\in H^\infty$, then 
		\[
	 	\begin{split}
	 	&	\sigma(A^\Theta_\phi)=\{\zeta\in\bbC : \inf_{z\in\bbD} (|\Theta(z)|+|\phi(z)-\zeta|)=0   \},\\
	 	&	\{\lambda\in\bbC: \lambda=\phi(z), \text{ where } z\in\bbD,\ \Theta(z)=0  \} \subset \sigma_p(A^\Theta_\phi),\\
	 &	\sigma_e(A^\Theta_\phi)=\{\zeta \in\bbC: \liminf_{z\in\bbD, |z|\to 1} (|\Theta(z) |+| \phi(z)-\zeta|=0 \}.
	 	\end{split}
		\]
\end{theorem}

As noted above, the class of TTOs is invariant by conjugation, and therefore we may obtain corresponding characterizations for coanalytic symbols. But what happens for more general TTOs?
Again a result in~\cite{AC} seems to be historically the first one. It deals with the essential spectrum of a TTO with continous symbol. More precisely, it states that 
\[
\sigma_e(A^\Theta_\phi)=\phi(\rho(\Theta)).
\]
There is a more extensive development of these ideas in~\cite{MR3203060}, which, in particular, studies the $C^*$-algebra generated by TTOs with continuous symbols.

The above characterization of the essential spectrum is extended in~\cite{Be} to symbols in $\Cont+H^\infty$. Since functions $\phi\in\Cont+H^\infty$ are defined only almost everywhere on $\bbT$, one should explain the meaning of the right hand side. The following is the precise statement of Bessonov's result; its form is similar to that of Theorem~\ref{th:spectral mapping theorem}.

\begin{theorem}\label{th:bessonov essential spectrum}
	Suppose $\phi\in\Cont+H^\infty$. Then
	\[
	\sigma_e(A^\Theta_\phi)=\{\zeta \in\bbC: \liminf_{z\in\bbD, |z|\to 1} (|\Theta(z) |+|\hat \phi(z)-\zeta|=0 \}
	\](for the definition of $\hat{\phi}$, see~\eqref{eq:poisson transform}).
\end{theorem}

It is harder to find criteria for invertibility of TTOs with nonanalytic symbols. The next part of the section uses embedding measures to obtain some partial results. We start with a statement which is essentially about bounded below TTOs.

\begin{theorem}\label{th:bb}
	Let $A$ be a (bounded) TTO, and let $\mu$ a complex measure such that $A=A_\mu$. 
	\begin{enumerate}
		\item
		If $A$ is bounded below, then $\iota_\mu$ is also bounded below, i.e. there exists $C>0$ (depending only on $\mu$ and $\Theta$) such that 
		\[\int_{\mathbb T}|f|^2 dm\leq C\int_{\mathbb T}|\iota_\mu(f)|^2d|\mu|\]
		for all $f\in K_\Theta$.  
		\item Suppose $A\in \TT(\Theta)^+$ and let $\mu\in \CC_2(\Theta)^+$ such that $A=A_\mu$. 
		The following assertions are equivalent:
		\begin{enumerate}
			\item the operator $A$ is invertible;
			\item there exists $C>0$ (depending only on $\mu$ and $\Theta$) such that 
			\[\int_{\mathbb T}|f|^2 dm\leq C\int_{\mathbb T}|\iota_\mu(f)|^2d\mu\]
			for all $f\in K_\Theta$.  
		\end{enumerate}
	\end{enumerate}

\end{theorem}

\begin{proof} (1)
	By definition of $A_\mu$, for all $f,g\in K_\Theta$, we have 
	\[ \left|\langle A_\mu f,g\rangle \right| = \left| \int_{\mathbb T} \iota_\mu (f)\overline{\iota_\mu (g)}d\mu \right| \leq  \int_{\mathbb T} |\iota_\mu (f)||\iota_\mu (g)|d|\mu|. \]
	The Cauchy--Schwarz inequality  implies that
	\[ \left|\langle A_\mu f,g\rangle \right| \leq \|\iota_\mu (f)\|_{L^2(|\mu|)} \|\iota_\mu (g)\|_{L^2(|\mu|)}\leq  \|\iota_\mu (f)\|_{L^2(|\mu|)} \|\iota_\mu\|\|g\|_2.\]
	Then, taking the supremum over all $g\in K_\Theta$ of unit norm, we get:
	\[\|A_\mu f\|_2\leq \|\iota_\mu\|\|\iota_\mu (f)\|_{L^2(|\mu|)} .\]
	Now, if $A_\mu$ is bounded below, there exists $C>0$ such that 
	\[  \|\iota_\mu (f)\|_{L^2(|\mu|)}\geq \frac{C}{\|\iota_\mu \|} \|f\|_2,\]
	and thus $\iota_\mu$ is bounded below.

(2)
	First, recall that since $A_\mu=A_\mu^*$, $A_\mu$ is invertible if and only if $A_\mu$ is bounded below.
	Thus	 (a) $\implies$ (b) follows  from part (1).    
	
	Conversely, assume that $\iota_\mu$ is bounded below and let $C$ be the constant defined  in $(b)$. It remains to check that $A_\mu$ is bounded below. 
	
	For a nonzero $f\in K_\Theta$, we have:
	\[\|A_\mu f\|_2\geq \left|  \langle A_\mu f, f/\|f\|_2\rangle  \right|=\frac{1}{\|f\|_2} \|\iota_\mu(f)\|^2_{L^2(\mu)}\geq \frac{1}{\|f\|_2}\frac{1}{C}\|f\|_2^2=\frac{\|f\|_2}{C},\]
	as expected.
\end{proof}

Volberg \cite{volberg} proved that given $\varphi\in L^\infty ({\mathbb T})$, and an inner function $\Theta$, the following are equivalent: 
\begin{itemize}
	\item there exist $C_1,C_2>0$ such that 
	\[  C_1\|f\|_2\leq \|f\|_{L^2(|\varphi|dm)} \leq C_2\|f\|_2, \]
	for all $f\in K_\Theta$;
	\item there exists $\delta>0$ such that \[\widehat{|\varphi|}(\lambda)+|\Theta(\lambda)|\geq \delta,\]
	for all $\lambda\in {\mathbb D}$.  
\end{itemize} 

 Volberg's result allows the translation of the embedding conditions in~Theorem~\ref{th:bb} into concrete functional inequalities, leading to the following statement.

\begin{theorem}\label{th:vol}
	Let $\varphi\in L^\infty ({\mathbb T})$ and let $\Theta$ be an inner function. 
	\begin{enumerate}
		\item If $A_\varphi^\Theta$ is bounded below, then there exists $\delta>0$ such that \[\widehat{|\varphi|}(\lambda)+|\Theta(\lambda)|\geq \delta,\]
		for all $\lambda\in {\mathbb D}$.  
		\item	If $\varphi\geq 0$, the following assertions are equivalent: 
		\begin{enumerate}
			\item The operator $A_\varphi^\Theta$ is invertible;
			\item there exists $\delta>0$ such that \[\widehat{|\varphi|}(\lambda)+|\Theta(\lambda)|\geq \delta,\]
			for all $\lambda\in {\mathbb D}$.  
		\end{enumerate}	
		
	\end{enumerate}
\end{theorem}

Denote by $\sigma_{ap}(T)$ the approximate point spectrum of a bounded operator $T$. 
\begin{corollary}
	Let $\varphi\in L^\infty ({\mathbb T})$ and let $\Theta$ be an inner function. Then 
	\[   \{ \mu\in {\mathbb C} :\inf_{\lambda\in {\mathbb D} }( \widehat{|\varphi-\mu|}(\lambda) +|\Theta (\lambda)|)=0 \}\subset \sigma_{ap}(A_\varphi^\Theta). \]
\end{corollary}

\section*{Acknowledgements}

The authors thank Richard Rochberg for some useful discussions. The authors 
were partially supported by French-Romanian project
LEA-Mathmode.

%

	 \newcommand{\noop}[1]{} \def\cprime{$'$} \def\cprime{$'$} \def\cprime{$'$}
	 \def\cprime{$'$}

\end{document}